\documentclass[twoside,leqno,%draft
]{amsart}
\usepackage{amssymb,amsmath,amsthm,latexsym,stmaryrd}

\newtheorem{thm}{Theorem}[section]

 \newtheorem{prop}[thm]{Proposition}
 \theoremstyle{definition}
 
 \theoremstyle{remark}

 \numberwithin{equation}{section}


\newcommand{\ie}{i.e.}
\newcommand{\f}{\phi}

\newcommand{\n}{\nabla}

\newcommand{\M}{(\mathcal{M},\allowbreak{}\f,\allowbreak{}\xi,\allowbreak{}\eta,g)}

\newcommand{\I}{\mathcal{I}}

\newcommand{\R}{\mathbb R}
\newcommand{\E}{\mathbb E}
\newcommand{\X}{\mathfrak X}
\newcommand{\F}{\mathcal{F}}

\newcommand{\MM}{\mathcal{M}}
\newcommand{\LL}{\mathfrak{L}}

\newcommand{\hatN}{\widehat{N}}
\newcommand{\tr}{{\rm tr}}
\newcommand{\ta}{\theta}

\newcommand{\om}{\omega}
\newcommand{\lm}{\lambda}

\newcommand{\sgn}{\mathrm{sgn}}

\newcommand{\al}{\alpha}

\newcommand{\ow}{\owedge}

\newcommand{\D}{\mathrm{d}}

\title
[Hyperspheres in Euclidean and Minkowski 4-spaces \ldots ]
{Hyperspheres in Euclidean and Minkowski 4-spaces
as almost paracontact almost paracomplex Riemannian manifolds}

\author[M. Manev]{Mancho  Manev}
\author[V. Tavkova]{Veselina  Tavkova}
\address[MM1, VT]{Department of Algebra and Geometry, Faculty of Mathematics and Informatics,
University of Plovdiv Paisii Hilendarski, 24, Tzar Asen St, 4000 Plovdiv,
Bulgaria}
\email{mmanev@uni-plovdiv.bg}
\email{vtavkova@uni-plovdiv.bg}

\address[MM2]{
Department of Medical Informatics, Biostatistics and E-Learning,
Faculty of Public Health, Medical University of Plovdiv,
15A, Vasil Aprilov Blvd, 4002 Plovdiv,
Bulgaria
}

\begin{document}

 %%% Insert a brief summary (50-150 words)
%\begin{abstract}
%%The object of investigations are almost paracontact almost paracomplex Riemannian manifolds of the lowest dimension.
%Almost paracontact almost paracomplex Riemannian manifolds of the lowest dimension are studied.
%There are constructed  on hyperspheres in 4-dimensional Euclidean and pseudo-Euclidean space, respectively.  The obtained manifolds are  characterised with respect to the classification used and their geometric  properties.
%%There are studied some geometrical characteristics
%% There are used two different
%%methods for construct  of the manifolds on hyperspheres in 4-dimensional spaces.
%
% %The constructed manifolds are characterised with respect to the used classification and their geometric  properties.
%\end{abstract}

\begin{abstract}
Almost paracontact almost paracomplex Riemannian manifolds of the lowest dimension are studied.
Such structures are constructed on hyperspheres in 4-dimensional spaces, Euclidean and pseudo-Euclidean, respectively.
The obtained manifolds are  studied and characterised in terms of the classification used and their geometric  properties.
\end{abstract}

%%%%%%%%%%%%%%%%%%%%%%%%%%%%%%%%%%%%%%%%%%%%%%%%%%%%%%%
\maketitle
%%%%%%%%%%%%%%%%%%%%%%%%%%%%%%%%%%%%%%%%%%%%%%%%%%%%%%%

\section{\textbf{Introduction}}\label{sec-intro}
 \vglue-10pt
 \indent

%In this work, we continue the study on almost paracontact almost paracomplex Riemannian manifolds.
In \cite{Sato76}, I. Sato introduced the notion of \emph{almost paracontact Riemannian structure}
on a differentiable manifold of arbitrary dimension so that this structure is compatible with a Riemannian metric such that the metric preserves the structure endomorphism on the paracontact distribution.
Later, other geometers (K. Matsumoto, T. Adati, T. Miyazawa, S. Sasaki) joined Sato in the initial development of the differential geometry of almost paracontact Riemannian manifolds (e.g. \cite{AdatMiya77, Sa80}).

Another type of structure-metric compatibility is known in addition to the above.
%Two types of metrics on an almost paracontact manifold are known and they differ from each other according to their compatibility with the almost paracontact structure.
%If the structure endomorphism induces an isometry with respect to the metric on the paracontact distribution of each tangent fibre,
%then it is said that the manifold has an \emph{almost paracontact Riemannian structure} (see, e.g., \cite{AdatMiya77, Sa80, Sato76}).
If the structure endomorphism induces an anti-isometry with respect to the metric on the paracontact distribution of each tangent fibre,
then it is said that the manifold has an \emph{almost paracontact metric structure} (see, e.g., \cite{KanKon, ZamNak}).
%In the other case, if the induced transformation is an anti-isometry, then it is said that the
%manifold is an \emph{almost paracontact metric manifold}, where the metric is pseudo-Riemannian of type $(n + 1, n)$
%This case is  studied by various authors

%A counterpart of this structure is the almost contact structure
%although the almost contact manifolds are odd-dimensional,
%while the almost paracontact manifolds can also be even-dimensional.
%The  differential geometry of almost contact manifolds is well studied, e.g. \cite{Blair} for the case with a compatible metric and \cite{GaMiGr} for the case of B-metric ???.

The restriction of the almost paracontact structure on the paracontact distribution is an almost product structure.
In \cite{Nav}, A.M. Naveira  gives a classification of Riemannian almost product
manifolds  with respect to the covariant derivative of the
almost product structure regarding the Levi-Civita connection of the Riemannian metric.

Almost paracontact Riemannian manifolds of Sasaki type $(n, n)$
%(i.e. with a traceless structure endomorphism)
are classified in \cite{ManSta01}.
For them, the induced almost product structure
on the paracontact distribution is traceless and it is called an \emph{almost paracomplex structure} (\cite{Lib52}).
These manifolds are necessarily odd-dimensional and
are called
\emph{almost paracontact almost paracomplex Riemannian manifolds} in \cite{ManVes18}.

An object of particular interest in our research is the case of the lowest dimension (which is three) of
almost paracontact almost paracomplex Riemannian manifolds. In this regard,
we study their properties in \cite{ManVes18, ManVes2, ManVes3}.

In the present work, we use two different approaches to construct
an almost paracontact almost paracomplex Riemannian manifold
on a hypersphere.
The first case is of a hypersphere in Euclidean space $\E^4$
and the second is of a time-like hypersphere in pseudo-Euclidean space
$\E^4_1$ (i.e. Minkowski space).
Similar research is made for almost contact B-metric hyperspheres in \cite{GaMiGr,HMan}.

The purpose of this paper
is to study the basic geometric characteristics of the considered manifolds.
The obtained  results will provide
explicit examples of the lowest dimension of the manifolds
under study and will contribute to the understanding of
their geometry.

The  paper is organized as follows.
In Sect.~\ref{sec-mfds}, we recall some necessary basic definitions and properties for the studied manifolds.
In Sect.~\ref{sec-R^4} and Sect.~\ref{sec-R^4_1}, we construct and  characterize such manifolds on hyperspheres in $\E^4$ and
$\E^4_1$, respectively.

%%%%%%%%%%%%%%%%%%%%%%%%%%%%%%%%%%%%%%%%%%

 \section{\textbf{Almost paracontact almost paracomplex Riemannian manifolds}}\label{sec-mfds}
 \vglue-10pt
 \indent

%Let $(\MM,\f,\xi,\eta)$ be an \emph{almost paracontact almost paracomplex Riemannian manifold}, \ie{}   $\MM$
%is an $(2n+1)$-dimensional real differentiable manifold equipped with an \emph{almost
%paracontact structure almost paracomplex} $(\f,\xi,\eta)$ consisting of  a tensor field
% $\f$  of type $(1,1)$ of the tangent bundle $T\MM$ of $\MM$, a vector field $\xi$, a 1-form
% $\eta$ and a Riemannian metric $g$  satisfying the following conditions: \cite{Sato76}, \cite{ManSta01}
%\begin{equation*}\label{str}
%\begin{array}{c}
%\f^2 = \I - \eta \otimes \xi,\quad \eta(\xi)=1,\quad
%\eta\circ\f=0,\quad \f\xi = 0,\quad \tr\f = 0,\\[4pt]
%g(\f x, \f y) = g(x,y) - \eta(x)\eta(y)
%\end{array}
%\end{equation*}
%where $\I$ denotes the identity on $T\MM$.

Let us consider an \emph{almost paracontact almost paracomplex Riemannian manifold} $(\MM,\f,\allowbreak{}\xi,\eta, g)$,
i.e. $\MM$ is a real differentiable manifold of dimension $(2n+1)$ equipped with an \emph{almost
paracontact  almost paracomplex structure} $(\f,\xi,\eta)$ and a Riemannian metric $g$.
Namely,
$\f$ is a tensor field of type $(1,1)$ (known as a paracontact endomorphism) of the tangent bundle $T\MM$ of $\MM$,  $\xi$ is a Reeb vector field and  $\eta$ is its dual 1-form, which together with $g$ satisfy the following conditions:
 \cite{Sato76}, \cite{ManSta01}
\begin{equation*}\label{str}
\begin{array}{c}
\f^2 = \I - \eta \otimes \xi,\quad \eta(\xi)=1,\quad
\eta\circ\f=0,\quad \f\xi = 0,\quad \tr\f = 0,\\[4pt]
g(\f x, \f y) = g(x,y) - \eta(x)\eta(y),
\end{array}
\end{equation*}
where $\I$ denotes the identity on $T\MM$.

Here and further $x$, $y$, $z$, $w$ will stand for arbitrary
elements of the Lie algebra $\X(\MM)$ of tangent vector fields on $\MM$ or vectors in the tangent space $T_p\MM$ at $p\in \MM$.

Let us denote the  Levi-Civita connection of $g$ by $\nabla $.
The fundamental tensor  $F$ of type (0,3) on  $(\MM,\f,\xi,\eta,g)$   is defined by
\begin{equation*}\label{F=nfi}
F(x,y,z)=g\bigl( \left( \nabla_x \f \right)y,z\bigr).
\end{equation*}
It has the following basic properties with respect to the structure
\begin{equation*}\label{F-prop}
\begin{array}{l}
F(x,y,z)=F(x,z,y)\\[4pt]
\phantom{F(x,y,z)}=-F(x,\f y,\f z)+\eta(y)F(x,\xi,z)
+\eta(z)F(x,y,\xi).
\end{array}
\end{equation*}
The relations of $\nabla \xi$ and $\nabla \eta$ with $F$ are as follows:
\begin{equation}\label{n_eta_F}
(\nabla_x \eta)(y)=g\left( \nabla_x \xi, y \right)=-F(x,\f y,\xi).
\end{equation}

Let $\left\{\xi;e_i\right\}$ $(i=1,2,\dots,2n)$ be a basis of
$T_p\MM$ at an arbitrary point $p\in \MM$ and $g^{ij}$  are the components of the inverse
matrix of $g$. Using this basis, the structure $(\f,\xi,\eta)$ and the metric $g$, the following 1-forms (known as Lee forms)
are associated with $F$:
\begin{equation*}\label{t}
\theta(z)=g^{ij}F(e_i,e_j,z),\quad
\theta^*(z)=g^{ij}F(e_i,\f e_j,z), \quad \omega(z)=F(\xi,\xi,z).
\end{equation*}
%These 1-forms are known also as the Lee forms of the considered manifolds.

In \cite{ManSta01}, a classification of almost paracontact almost paracomplex
Riemannian manifolds is made with respect to basic properties of $F$ with respect to the tensor structure of the studied manifold. %consisting of eleven basic classes $\F^{s}$, $s\in\{1,2,\dots, 11\}$.
This classification consists of 11 basic
classes $\F_1$, $\F_2$, $\dots$, $\F_{11}$.
Furthermore,
the components $F^{s}$ $(s\in\{1,2,\dots,11\})$ of $F$, which correspond to the classes $\F_{s}$, are determined in \cite{ManVes18}.
The latter approach provides an alternative way to determine the basic classes of the considered classification.
Namely, the manifold $\M$ belongs to $\F_{s}$
if and only if the equality $F=F^s$ is valid.
As a corollary we have the following. A manifold of the studied type
belongs to a direct sum of two or more basic classes, i.e.
$\M\in\F_i\oplus\F_j\oplus\cdots$, if and only if the
tensor $F$ on $\M$ is the sum of the corresponding components
$F^i$, $F^j$, $\ldots$ of $F$, i.e. the following condition is
satisfied $F=F^i+F^j+\cdots$.

%\colorbox{yellow}{Dali se nalaga da vyvezhdame $\F_0$?} The intersection of the basic classes is the special class $\F_0$
%determined by the condition $F(x,y,z)=0$. Then $\F_0$ is the
%class of the studied manifolds with $\n$-parallel
%structures, i.e. $\n\f=\n\xi=\n\eta=\n g=\n \tg=0$.
%

Let  $(\MM,\f,\xi,\eta,g)$ have the lowest dimension
(i.e. $\dim{\MM}=3$)  and let the set of vectors
$\{e_0, e_1, e_2\}$ be a $\f$-basis of $T_p\MM$ which satisfies the following
conditions:
\begin{equation}\label{strL}
\begin{array}{l}
\f e_0=0,\quad \f e_1=e_{2},\quad \f e_{2}= e_1,\quad \xi=
e_0,\quad \\[4pt]
\eta(e_0)=1,\quad \eta(e_1)=\eta(e_{2})=0,
\end{array}
\end{equation}
\begin{equation}\label{gL}
  g(e_i,e_j)=\delta_{ij},\qquad i,j\in\{0,1,2\}.
\end{equation}

According to \cite{ManVes18}, the components ${F_{ijk}=F(e_i,e_j,e_k)}$, ${\ta_k=\ta(e_k)}$, ${\ta^*_k=\ta^*(e_k)}$ and ${\om_k=\om(e_k)}$ of $F$, $\ta$, $\ta^*$ and $\om$, respectively,  with respect to the $\f$-basis $\left\{e_0,e_1,e_2\right\}$ are determined as  follows:
\begin{equation*}\label{t3}
\begin{array}{c}
	\begin{array}{ll}
		\ta_0=F_{110}+F_{220},\quad & \ta_1=F_{111}=-F_{122}=-\ta^*_2,\\[4pt]
		\ta^*_0=F_{120}+F_{210}, \quad &\ta_2=F_{222}=-F_{211}=-\ta^*_1,\\[4pt]
	\end{array}\\
	\begin{array}{lll}
		\om_0=0,  \qquad & \om_1=F_{001},\qquad & \om_2=F_{002}.
	\end{array}
\end{array}
\end{equation*}
Hence, the components $F^s$, $s\in\{1,2,\dots, 11\}$, of $F$ on $(\MM,\f,\xi,\eta,g)$  in  the corresponding basic classes $\F_s$ have the following form: \cite{ManVes18}
\begin{equation}\label{Fi3}
\begin{array}{l}
F^{1}(x,y,z)=\left(x^1\ta_1-x^2\ta_2\right)\left(y^1z^1-y^2z^2\right); \\[4pt]
F^{2}(x,y,z)=F^{3}(x,y,z)=0;
\\
F^{4}(x,y,z)=\frac{\ta_0}{2}\Bigl\{x^1\left(y^0z^1+y^1z^0\right)
+x^2\left(y^0z^2+y^2z^0\right)\bigr\};\\[4pt]
F^{5}(x,y,z)=\frac{\ta^*_0}{2}\bigl\{x^1\left(y^0z^2+y^2z^0\right)
+x^2\left(y^0z^1+y^1z^0\right)\bigr\};\\[4pt]
F^{6}(x,y,z)=F^{7}(x,y,z)=0;\\[4pt]
F^{8}(x,y,z)=\lm\bigl\{x^1\left(y^0z^1+y^1z^0\right)
-x^2\left(y^0z^2+y^2z^0\right)\bigr\},\\[4pt]
%&
\hspace{38pt} \lm=F_{110}=-F_{220}
;\\[4pt]
F^{9}(x,y,z)=\mu\bigl\{x^1\left(y^0z^2+y^2z^0\right)
-x^2\left(y^0z^1+y^1z^0\right)\bigr\},\\[4pt]
%&
\hspace{38pt} \mu=F_{120}=-F_{210}
;\\[4pt]
F^{10}(x,y,z)=\nu x^0\left(y^1z^1-y^2z^2\right),\quad
\nu=F_{011}=-F_{022}
;\\[4pt]
F_{11}(x,y,z)=x^0\bigl\{\om_{1}\left(y^0z^1+y^1z^0\right)
+\om_{2}\left(y^0z^2+y^2z^0\right)\bigr\},
\end{array}
\end{equation}
where the decompositions $x=x^ie_i$, $y=y^ie_i$, $z=z^ie_i$
%are arbitrary vectors in $T_p\MM$, $p\in \MM$,
with respect to $\left\{e_0,e_1,e_2\right\}$ are used.

By virtue of \eqref{Fi3}, it is determined in \cite{ManVes18} that the studied 3-dimensional manifolds can belong only to
the basic classes
$\F_1$, $ \F_4$, $\F_5$,  $\F_8$, $\F_9$, $\F_{10}$,  $\F_{11}$ and their direct sums.

The Nijenhuis tensor $N$ of the structure $(\f,\xi,\eta)$ is defined by the equality
$N(x,y) = [\f,\f](x,y) - \D\eta(x,y) \xi$,
where the Nijenhuis torsion of $\f$ is determined by
$[\f,\f](x,y)=[\f x,\f y]+\f^2[x,y]-\f[\f x,y]-\f[x,\f y]$
and $\D\eta$ is the exterior derivative of $\eta$ given by
$\D\eta (x,y)=(\nabla _x\eta)y-(\nabla _y\eta)x$.
The corresponding tensor of type (0,3) of the Nijenhuis tensor on $\M$ is defined by the
equality  $N(x,y,z)=g\left(N(x,y),z\right)$.
According to \cite{ManVes18},  we express $N$ in terms of $F$ as follows:
\begin{equation}\label{N}
\begin{array}{ll}
N(x,y,z)=F(\f x,y,z)-F(\f y,x,z)-F(x,y,\f z)+F(y,x,\f z)\\[4pt]
\phantom{N(x,y,z)=}+\eta(z)\left\{F(x,\f y,\xi)-F(y,\f
x,\xi)\right\}.
\end{array}
\end{equation}
The associated Nijenhuis tensor  $\hatN$ of the structure
$(\f,\xi,\eta,g)$ is defined by the following way
$\hatN(x,y) = \{\f,\f\}(x,y) {-} ( \mathcal \LL_{\xi}g)(x,y) \xi$.
In the latter equality, $\{\f ,\f\}$ is the symmetric tensor of type $(1,2)$ determined by
$\{\f ,\f\}(x,y)=\{\f x,\f y\}+\f^2\{x,y\}-\f\{\f x,y\}-\f\{x,\f y\}$
for $\{x,y\} = \nabla_x y +  \nabla_y x$ and $\LL_{\xi}g$ is  the Lie derivative  of $g$ along $\xi$
expressed by $\left(\LL_{\xi}g\right)(x,y)=(\nabla _x\eta)y + (\nabla _y\eta)x$.
The corresponding tensor of type (0,3) of the associated Nijenhuis tensor is defined by
$\hatN(x,y,z)=g\left(\hatN(x,y),z\right)$. In \cite{ManVes18}, we express $\hatN$  by  $F$ as
follows:
\begin{equation}\label{HN}
\begin{array}{ll}
\hatN(x,y,z)=F(\f x,y,z)+F(\f y,x,z)-F(x,y,\f z)-F(y,x,\f z)\\[4pt]
\phantom{\hatN(x,y,z)=}
+\eta(z)\left\{F(x,\f y,\xi)+F(y,\f x,\xi)\right\}.
\end{array}
\end{equation}

The curvature tensor $R$ of type $(1,3)$ for $\nabla$ is defined as usually by $R=\left[\n,\n\right]-\n_{[\ ,\ ]}$.
The corresponding $(0,4)$-tensor is denoted by the same letter and it is given  by  $R(x,y,z,w)=g(R(x,y)z,w)$.
%It has the following properties:
%\begin{equation}\label{R}
%\begin{array}{c}
%    R(x,y,z,w)=-R(y,x,z,w)=-R(x,y,w,z), \\[4pt]
%R(x,y,z,w)+R(y,z,x,w)+R(z,x,y,w)=0.
%\end{array}
%\end{equation}

The Ricci tensor $\rho$ and the scalar curvature $\tau$ for $R$ as well as
their associated quantities are determined respectively by:
\begin{equation}\label{rhotau}
\begin{array}{ll}
    \rho(y,z)=g^{ij}R(e_i,y,z,e_j),\qquad &
    \tau=g^{ij}\rho(e_i,e_j),\\[4pt]
    \rho^*(y,z)=g^{ij}R(e_i,y,z,\f e_j),\qquad &
    \tau^*=g^{ij}\rho^*(e_i,e_j).
\end{array}
\end{equation}

Moreover, we use the Kulkarni-Nomizu
product $g\ow  h$ of two $(0,2)$-tensors $g$ and $h$ defined by
\[
\begin{array}{l}
\left(g\ow h\right)(x,y,z,w)=g(x,z)h(y,w)-g(y,z)h(x,w)\\[4pt]
\phantom{\left(g\ow h\right)(x,y,z,w)}
+g(y,w)h(x,z)-g(x,w)h(y,z).
\end{array}
\]
Obviously,   $g\ow h$ has the basic properties of $R$
if and only if $g$ and $h$ are symmetric.

Let $\al$ be a non-degenerate 2-plane in $T_p\MM$, $p \in \MM$, having a basis  $\{x,y\}$.
The sectional curvature $k(\al;p)$  is determined by
\begin{equation}\label{sect}
k(\al;p)=-\frac{2R(x,y,y,x)}{(g\ow g)(x,y,y,x)}.
\end{equation}

%We consider the following special 2-planes with respect to the manifold structure.
%A 2-plane is called a \emph{$\f$-holomorphic section}
%(respectively, a \emph{$\xi$-sec\-tion})
%if $\al= \f\al$ (respectively, $\xi \in \al$).
%\colorbox{yellow}{Dali da kazvame za tezi special 2-planes, kato ne se izpolzvat realno?}

%\colorbox{yellow}{Njakoi ot nomeriranite ravenstva v \S2 ne sa citirani. Ili da se citirat, }\\
%\colorbox{yellow}{kydeto trjabva, ili da se mahne nomera.}

\section{\textbf{A hypersphere with the studied structure in Euclidean 4-space}}\label{sec-R^4}
 \vglue-10pt
 \indent

 Let $\E^{4}$ be the Euclidean space $\left(\R^4,\langle\cdot,\cdot\rangle\right)$, where $\langle\cdot,\cdot\rangle$ is the usual Euclidean inner product determined by
\begin{equation*}
  \langle x,y\rangle=x^1y^1+x^2y^2+x^3y^3+x^4y^4
\end{equation*}
for $x(x^1,x^2,x^3,x^4)$, $y(y^1,y^2,y^3,y^4)$ from $\R^{4}$.

Then, we consider a hypersphere $S_1$ in $\E^{4}$ at the origin with a real radius $r$ identifying an arbitrary point $p$ in $\E^{4}$ with its position vector $z$, \ie
\begin{equation}\label{defS1}
S_1:\quad \langle z,z\rangle=r^2.
\end{equation}
It has the following parametrization
\begin{equation*}
{z}(r\cos u^1\cos u^2, r\cos u^1\sin u^2, r\sin u^1\cos u^0,  r\sin u^1\sin u^0),
\end{equation*}
where $u^0,u^1,u^2$ are real parameters such as $u^0, u^1, u^2 \in [0; 2\pi)$, $u^1\neq\frac{k\pi}{2}$ for $k \in \{0,1,2,3\}$.
Consequently, the local basic vectors $\partial_i=\frac{\partial z}{\partial{u^i}}$, $i \in \{0,1,2\}$
have the following inner products:
\begin{equation*}
\begin{array}{l}
 \langle\partial_0,\partial_0\rangle=r^2 \sin^2u^1, \quad
   \langle\partial_1,\partial_1\rangle=r^2, \quad
  \langle\partial_2,\partial_2\rangle=r^2 \cos^2u^1, \quad\\[4pt]
  \langle\partial_i,\partial_j\rangle=0, \; i\neq j.
\end{array}
\end{equation*}
Substituting $e_i=\frac{1}{\sqrt{\langle\partial_i,\partial_i\rangle}}\partial_i$, $i \in \{0,1,2\}$, we obtain an orthonormal basis $\{e_i\}$, $i \in \{0,1,2\}$ as follows
\begin{equation}\label{S1-eidi}
\begin{array}{l}
{e_0}=\frac{\varepsilon_2}{r\sin u^1}\partial_0, \qquad
{e_1}=\frac{1}{r}\partial_1, \qquad
{e_2}=\frac{\varepsilon_1}{r\cos u^1}\partial_2,
\end{array}
\end{equation}
where $\varepsilon_1=\sgn (\cos u^1)$, $\varepsilon_2=\sgn (\sin u^1)$.

Next, we introduce an almost paracontact almost paracomplex structure $(\f,\xi,\eta)$ on
$S_1$ determined as shown  in \eqref{strL}.
The metric $g$ on the hypersurface is the
restriction of $\langle\cdot,\cdot\rangle$ on $S_1$.
Therefore, $\{e_i\}$, $i \in \{0,1,2\}$ is an orthonormal $\f$-basis with respect to $g$ on  $T_pS_1$ at $p \in S_1$, i.e.  \eqref{gL} is satisfied.
Thus, we obtain for $(S_1,\f,\xi,\eta,g)$ the following
\begin{prop}
The manifold $(S_1,\f,\xi,\eta,g)$ is a 3-dimensional almost paracontact almost paracomplex Riemannian manifold.
\end{prop}

Using \eqref{S1-eidi}, we calculate the following commutators of the basic vectors $e_i$:
\begin{equation}\label{S1-com}
\begin{array}{l}
 [e_0,e_1]=\frac{1}{r}\cot u^1 {e_0}, \qquad [e_0,e_2]=0, \qquad [e_1,e_2]=\frac{1}{r}\tan u^1 {e_2}.
\end{array}
\end{equation}
According to the latter equations and the Koszul equality for $\n$ of $g$, \ie
\begin{equation}\label{Koz}
2g\left(\n_{E_i}E_j,E_k\right)
=g\left([E_i,E_j],E_k\right)+g\left([E_k,E_i],E_j\right)
+g\left([E_k,E_j],E_i\right),
\end {equation}
we obtain  the components of the  covariant derivatives of $e_i$ with respect to $\n$:
\begin{equation}\label{S1-nij}
\begin{array}{ll}
\n_{e_0}e_0=-\frac{1}{r}\cot u^1 {e_1}, \qquad &\n_{e_0}e_1=\frac{1}{r}\cot u^1 {e_0},\\[4pt]
\n_{e_2}e_1=-\frac{1}{r}\tan u^1 {e_2}, \qquad &\n_{e_2}e_2=\frac{1}{r}\tan u^1 {e_1}
\end{array}
\end{equation}
and the remaining $\n_{e_i}e_j$ are zero.

%Taking into accont   \eqref{strL}, \eqref{gL} and \eqref{S1-nij}, we get the value of the square norm of $\n \f$:
%\begin{equation}\label{S1-nf}
%\norm{\nabla \f}=\frac{1}{r^2}\cot ^2 u^1( e_{0}^2+  e_{2}^2)+\frac{4}{r^2}\tan ^2 u^1( e_{1}^2+  e_{2}^2).
%\end{equation}

Bearing in mind \eqref{strL}, \eqref{gL} and \eqref{S1-nij}, we obtain the following components $F_{ijk}$ of $F$ with respect to the basis $\{e_i\}$, $i \in \{0,1,2\}$:
\begin{equation}\label{S1-Fijk}
  F_{002}=F_{020}=\frac{1}{r}\cot u^1, \qquad   F_{211}=-F_{222}=\frac{2}{r}\tan u^1
\end{equation}
and the other components $F_{ijk}$ are zero.

According to  \eqref{N}, \eqref{HN} and \eqref{S1-Fijk}, we determine the  basic components $N_{ijk}=N(e_i,e_j,e_k)$ of the Nijenhuis tensor and $\widehat{N}_{ijk}=\widehat{N}(e_i,e_j,e_k)$ of  its associated tensor.
The non-zero of them are:
\begin{equation*}
\begin{array}{l}
N_{010}=-N_{100}=\frac{1}{r}\cot u^1, \\[4pt]
\widehat{N}_{221}=\widehat{N}_{111}=-\widehat{N}_{122}=-\widehat{N}_{212}=\frac{4}{r}\tan u^1, \\[4pt]
\widehat{N}_{001}=-\frac{2}{r}\cot u^1, \qquad \widehat{N}_{010}=\widehat{N}_{100}=\frac{1}{r}\cot u^1.
\end{array}
\end{equation*}
%From \eqref{snf}, the square norms of $N$ and $\widehat{N}$ have the following form
%\begin{equation}\label{S1-NNhat}
%\begin{array}{l}
%\norm{N}=\frac{2}{r^2}\cot^2 u^1, \\[4pt]
%\norm{\widehat{N}}=\frac{6}{r^2}\cot^2 u^1+\frac{64}{r^2}\tan^2 u^1.
%\end{array}
%\end{equation}

Using \eqref{Fi3} and \eqref{S1-Fijk}, we get the equality
\begin{equation*}\label{Fexa}
%\begin{array}{l}
 F(x,y,z)=\frac{2}{r}\tan u^1x^2\left(y^1z^1-y^2z^2\right)+\frac{1}{r}\cot u^1 x^0\left(y^0z^2+y^2z^0\right).
%\end{array}
\end{equation*}
By virtue of  the latter equality, we establish that $F$ has  the following form:
\begin{equation}\label{S1-F1+F11}
F(x,y,z)=F^1(x,y,z)+F^{11}(x,y,z),
\end{equation}
where $F^1$ and $F^{11}$ are the components of $F$ for the basic classes $\F_1$ and $\F_{11}$, respectively.
Therefore, we have  the following non-zero components of $F^1$ and $F^{11}$ with respect to $\{e_i\}$, $i \in \{0,1,2\}$:
\begin{equation}\label{S1-Fijk59}
\begin{array}{l}
F^1_{211}=-F^1_{222}=-\ta_2=\frac{2}{r}\tan u^1,\\[4pt]
F^{11}_{002}=F^{11}_{020}=\om_2=\frac{1}{r}\cot u^1.
\end{array}
\end{equation}
Let us note that  the  components of $F^1$ and $F^{11}$ from the above  are non-zero for all values of $u^1$ in its domain.

Next, using  \eqref{n_eta_F},  \eqref{S1-F1+F11} and  \eqref{S1-Fijk59},  we find the following:
\begin{equation*}\label{S1-d-eta}
N=-\D\eta \otimes \xi, \qquad \n_\xi \xi\neq 0,
\end{equation*}
which support the obtained results in \cite{ManVes18}.

Bearing in mind  \eqref{gL}, \eqref{S1-com}, \eqref{S1-nij} and the definition equality of $R$,  we obtain  the components $R_{ijk\ell}=R(e_i,e_j,e_k,e_\ell)$ of $R$ with respect to $\{e_i\}$, $i \in \{0,1,2\}$. The first of them are:
\begin{equation}\label{S1-Rijkl}
  R_{0101}=R_{0202}=R_{1212}=-\frac{1}{r^2}.
\end{equation}
The rest of the non-zero components of $R$ are determined by \eqref{S1-Rijkl}
and the basic symmetries of $R$ and its first Bianchi identity.

According to \eqref{gL}, \eqref{rhotau} and \eqref{S1-Rijkl}, we obtain  the  components $\rho_{jk}=\rho(e_j,e_k)$ and $\rho^*_{jk}=\rho^*(e_j,e_k)$ of the
Ricci tensor $\rho$ and  the $*$-Ricci tensor  $\rho^*$, respectively, as well as the values
of the scalar curvature $\tau$ and its associated quantity $\tau^*$ as follows:
\begin{equation}\label{S1-rhotau}
\begin{array}{l}
\rho_{00}=\rho_{11}=\rho_{22}=\frac{2}{r^2}, \quad \rho^{*}_{12}=\rho^{*}_{21}=-\frac{1}{r^2},\\[4pt]
\tau=\frac{6}{r^2}, \qquad \tau^{*}=0.
\end{array}
\end{equation}
Futhermore, from  \eqref{gL}, \eqref{sect} and \eqref{S1-Rijkl}, we get the basic sectional curvatures $k_{ij}=k(e_i,e_j)$ determined by the basis $\{e_i,e_j\}$ of the corresponding 2-plane:
\begin{equation}\label{S1-kij}
k_{01}=k_{02}=k_{12}=\frac{1}{r^2}.
\end{equation}

Taking into account \eqref{gL}, \eqref{S1-Rijkl} and \eqref{S1-kij}, we get the  form of the curvature tensor as follows
\begin{equation}\label{Rpi1}
R(x,y,z,w)
%=\frac{1}{r^2}\{g(y,z)g(x,w)-g(x,z)g(y,w)\}.
=-\frac{1}{2r^2}(g\ow g)(x,y,z,w).
\end{equation}
%\begin{equation}\label{Rpi1}
%R(x,y,z,w)=\frac{1}{r^2}\pi_1 (x,y,z,w),
%\end{equation}
%where $\pi_1 (x,y,z,w)=g(y,z)g(x,w)-g(x,z)g(y,w)$ is  the essential curvature-like tensor of type $(0,4)$.
%\colorbox{yellow}{V \S1 izpolzvash Kulkarni-Nomizu product, a tuk $\pi_1$. Uednakvi gi!}

According to the obtained results from the above, we have the following
\begin{thm}
Let $(S_1,\f,\xi,\eta,g)$ be the  hypersphere in the  Euclidean 4-space $\E^{4}$
equipped with an almost paracontact almost paracomplex structure and a Riemannian metric
 defined by \eqref{defS1}, \eqref{strL} and \eqref{gL}.
Then, the manifold $(S_1,\f,\xi,\eta,g)$  has the following properties:
\begin{enumerate}
\item
it is in the class $\F_1\oplus\F_{11}$  but does not belong to either $\F_1$ or $\F_{11}$;
\item
it has a positive scalar curvature;
\item
it is $*$-scalar flat;
    \item
it is a space-form of positive constant sectional curvature.
\end{enumerate}
\end{thm}
\begin{proof}
We establish the truthfulness of assertion (1)  using \eqref{S1-F1+F11} and \eqref{S1-Fijk59}.
Conclusions (2) and (3) are consequences of \eqref{S1-rhotau}, whereas (4)  follows from \eqref{Rpi1}.
\end{proof}

\section{\textbf{A hypersphere with the studied structure in Minkowski 4-space}}\label{sec-R^4_1}
 \vglue-10pt
 \indent

In this section we consider the pseudo-Euclidean space $\E^{4}_{1}$,
\ie{}
the real 4-space $\R^4$ equipped with  the following Lorentzian inner product %$\langle\cdot,\cdot\rangle$
\begin{equation}\label{inprod3-1}
  \langle x,y\rangle=x^1y^1+x^2y^2+x^3y^3-x^4y^4
\end{equation}
for arbitrary vectors $x(x^1,x^2,x^3,x^4)$ and $y(y^1,y^2,y^3,y^4)$ in $\R^{4}$.

In a similar manner as in the previous section, we define the following hypersphere $S_2$ in $\E^{4}_1$
at the origin with real radius $r$,
%identifying an arbitrary point $p$ in $\E^{4}_1$ with its position vector $z$
%determined by the condition
\[
S_2:\quad
\langle z,z\rangle=-r^2.
\]
It has the following parametrization
\begin{equation*}
{z}(r\sinh u^1\cos u^2, r\sinh u^1\sin u^2, r\cosh u^1\sinh u^3,  r\cosh u^1\cosh u^3),
\end{equation*}
where $u^1,u^2,u^3$ are real parameters such as $u^1\in (-\infty;0)\cup(0;+\infty)$, $u^2 \in [0;2\pi)$,
$u^3\in (-\infty;+\infty)$.

Therefore, for the local basic vectors $\partial_i=\frac{\partial z}{\partial{u^i}}$, $i\in\{1,2,3\}$, we obtain the following
\begin{equation*}
\begin{array}{l}
  \langle\partial_1,\partial_1\rangle=r^2, \quad
  \langle\partial_2,\partial_2\rangle=r^2 \sinh^2u^1, \quad
    \langle\partial_3,\partial_3\rangle=r^2 \cosh^2u^1,\\[4pt]
  \langle\partial_i,\partial_j\rangle=0, \; i\neq j.
\end{array}
\end{equation*}

Then, we substitute  $e_{i-1}=\frac{1}{\sqrt{\left|\langle\partial_i,\partial_i\rangle\right|}}\partial_i$ and get the orthonormal basis $\{e_i\}$, $i \in \{0,1,2\}$, as follows
\begin{equation*}\label{S2-eidi}
\begin{array}{l}
{e_0}=\frac{1}{r}\partial_1, \qquad
{e_1}=\frac{1}{r\sinh u^1}\partial_2, \qquad
{e_2}=\frac{1}{r\cosh u^1}\partial_3.
\end{array}
\end{equation*}
%where $\varepsilon_1=\sgn (\cos u^1)$, $\varepsilon_2=\sgn (\sin u^1)$.

In the same way as in the previous section, here we equip $S_2$ with an almost paracontact almost paracomplex structure and a Riemannian metric defined by  \eqref{strL} and \eqref{gL}, respectively.
Thus, for the obtained manifold $(S_2,\f,\xi,\eta,g)$, we have the following
\begin{prop}
The manifold $(S_2,\f,\xi,\eta,g)$ is a 3-dimensional almost paracontact almost paracomplex Riemannian manifold.
\end{prop}

By similar considerations  as for $S_1$, we obtain the following:
\begin{equation}\label{S2-com}
\begin{array}{c}
[e_0,e_1]=-\frac{1}{r}\coth u^1 {e_1}, \qquad [e_0,e_2]=-\frac{1}{r}\tanh u^1 {e_2}, \qquad  [e_1,e_2]=0.
\end{array}
\end{equation}
Taking into account \eqref{Koz} and \eqref{S2-com} we get:
\begin{equation}\label{S2-nij}
\begin{array}{ll}
\n_{e_1}e_0=\frac{1}{r}\coth u^1 {e_1}, \qquad &\n_{e_2}e_0=\frac{1}{r}\tanh u^1 {e_2},\\[4pt]
\n_{e_1}e_1=-\frac{1}{r}\coth u^1 {e_0}, \qquad &\n_{e_2}e_2=-\frac{1}{r}\tanh u^1 {e_0}.
\end{array}
\end{equation}
Bearing in mind \eqref{strL}, \eqref{gL} and  \eqref{S2-nij}, we compute the components $F_{ijk}$ of $F$. The
non-zero of them are:
\begin{equation}\label{S2-Fijk}
\begin{array}{c}
F_{102}=F_{120}=-\frac{1}{r}\coth u^1, \qquad  F_{201}=F_{210}=-\frac{1}{r}\tanh u^1.
\end{array}
\end{equation}
Then, applying  \eqref{N}, \eqref{HN} and \eqref{S2-Fijk}, we calculate the components $N_{ijk}$ and $\widehat{N}_{ijk}$  as follows:
\begin{equation*}
\begin{array}{l}
N_{101}=-N_{011}=N_{022}=-N_{202}=\frac{2}{r\sinh 2u^1}, \\[4pt]
\widehat{N}_{101}=\widehat{N}_{011}=-\widehat{N}_{202}=-\widehat{N}_{022}=\frac{2}{r\sinh 2u^1}, \\[4pt]
\widehat{N}_{110}=\widehat{N}_{220}=-\frac{2}{r}(\coth u^1 + \tanh u^1).
\end{array}
\end{equation*}
By vitue of   \eqref{Fi3} and \eqref{S2-Fijk}, we establish the following equality
\begin{equation}\label{S2-F5+F9}
F(x,y,z)=(F^5+F^9)(x,y,z),
\end{equation}
where $F^5$ and $F^9$
are the components of $F$ corresponding to the basic classes $\F_5$ and $\F_9$.
The non-zero components of $F^5$ and $F^9$ with respect to the basis $\{e_0,e_1,e_2\}$ are the following
\begin{equation}\label{S2-Fijk59}
\begin{array}{l}
F^5_{102}=F^5_{120}=F^5_{201}=F^5_{210}=\frac{1}{2}\ta^*_0=-\frac{1}{2r}(\coth u^1+\tanh u^1),\\[4pt]
F^9_{102}=F^9_{120}=-F^9_{201}=-F^9_{210}=\mu=\frac{1}{2r}(\tanh u^1-\coth u^1).
\end{array}
\end{equation}
Taking into account \eqref{n_eta_F}, \eqref{S2-F5+F9} and \eqref{S2-Fijk59},  we get
\begin{equation}\label{S2-d-eta}
\D\eta=0, \qquad \n_\xi \xi=0,
\end{equation}
 which support the obtained results in \cite{ManVes18}.

Bearing in mind \eqref{gL}, \eqref{S2-com} and \eqref{S2-nij}, we calculate the components $R_{ijkl}$ of $R$. The
non-zero of them are determined by the basic symmetries of $R$ and the following
\begin{equation}\label{S2-Rijkl}%\label{S2-Rijkl}
\begin{array}{c}
R_{0101}=R_{0202}=R_{1212}=\frac{1}{r^2}.
\end{array}
\end{equation}
Using \eqref{gL}, \eqref{rhotau} and \eqref{S2-Rijkl}, we obtain the basic components  $\rho_{jk}$ and $\rho^*_{jk}$  as well as the values of
 $\tau$ and  $\tau^*$:
\begin{equation}\label{S2-rhotau}
\begin{array}{ll}
\rho_{00}=\rho_{11}=\rho_{22}=-\frac{2}{r^2}, \qquad & \rho^{*}_{12}=\rho^{*}_{21}=\frac{1}{r^2},\\
\tau=-\frac{6}{r^2}, \qquad & \tau^{*}=0.
\end{array}
\end{equation}
Bearing in mind \eqref{gL}, \eqref{sect} and \eqref{S2-Rijkl}, we compute the basic sectional curvatures $k_{ij}$
with respect to the basis $\{e_0,e_1,e_2\}$
as follows
\begin{equation}\label{S2-kij}
k_{01}=k_{02}=k_{12}=-\frac{1}{r^2}.
\end{equation}

Bearing in mind  \eqref{gL}, \eqref{S2-Rijkl} and \eqref{S2-kij}, we get the  form of the curvature tensor in the following way
%\begin{equation}\label{Rpi2}
%R(x,y,z,w)=-\frac{1}{r^2}\pi_1 (x,y,z,w).
%\end{equation}
\begin{equation}\label{Rpi2}
R(x,y,z,w)=%-\frac{1}{r^2}\{g(y,z)g(x,w)-g(x,z)g(y,w)\}=
\frac{1}{2r^2}(g\ow g)(x,y,z,w).
\end{equation}

By virtue the obtained results from the above, we obtain the following
\begin{thm}
Let $(S_2,\f,\xi,\eta,g)$ be the time-like sphere in the Minkowski 4-space $\E^{4}_1$ equipped  with an almost paracontact almost paracomplex structure and a Riemannian metric defined by \eqref{inprod3-1}, \eqref{strL} and \eqref{gL}.
Then, the manifold $(S_2,\f,\xi,\eta,g)$ has the following properties:
\begin{enumerate}
\item
it is in the class $\F_5\oplus\F_9$
but does not belong to either $\F_5$ or $\F_9$;
\item
it has a closed 1-form $\eta$ and geodesic integral curves of $\xi$;
\item
it has a negative scalar curvature;
\item
it is $*$-scalar flat;
% \item
%it has negative sectional curvatures of the basic $\xi$-sect\-ions;
%\item
%it has a negative scalar curvature of the basic $\f$-holomorphic section;
\item
it is a space-form of negative constant sectional curvature.
\end{enumerate}
\end{thm}

\begin{proof}
We establish the truthfulness of assertion (1) using   \eqref{S2-F5+F9} and \eqref{S2-Fijk59}.
Statements (2), (3)--(4), (5) follow directly from \eqref{S2-d-eta},  \eqref{S2-rhotau}, \eqref{Rpi2}, respectively.
\end{proof}

\bigskip

\subsection*{Acknowledgment}
The authors were supported by project
of the Scientific Research Fund,
University of Plovdiv Paisii Hilendarski, Bulgaria.

 \end{document}